\documentclass[11pt]{article}

\newcommand\version{June 16, 2026}

\usepackage{amsmath,amssymb,amsthm,mathtools}
\usepackage{url}

\newtheorem{theorem}{Theorem}
\newtheorem{lemma}[theorem]{Lemma}

\theoremstyle{remark}

\newtheorem{remark}[theorem]{Remark}

\newcommand{\E}{\mathcal E}
\newcommand{\loc}{{\rm loc}}
\newcommand{\norm}[1]{\left\|#1\right\|}

\newcommand{\R}{\mathbb R}
\newcommand{\C}{\mathbb C}
\newcommand{\eps}{\varepsilon}
\newcommand{\csch}{\operatorname{csch}}
\newcommand{\q}{\mathfrak q}

\title{The sharp extension norm for a planar sector}
\author{Rupert L. Frank\footnote{LMU Munich, r.frank@lmu.de} \footnote{Munich Center for Quantum Science and Technology} \footnote{Caltech}
\and Paata Ivanisvili\footnote{University of California, Irvine, pivanisv@uci.edu}}
\date{\version}

\begin{document}
\maketitle

\begin{abstract}
We compute explicitly the infimum of the norms of $W^{1,2}$-extension operators for planar sectors and exhibit an extension operator attaining this infimum. This solves an open problem posed by Maz'ya.
\end{abstract}


\section{Introduction and main result}

The notion of an extension operator is fundamental in the theory of Sobolev spaces \cite{AdamsFournier,MazyaSobolev,Leoni} and in its applications \cite{Evans,FrLaWe}. The problem of sharp constants for extension operators has, as far as we know, hardly received any attention so far. Perhaps this is the reason why Maz'ya included into his list of `Seventy Five (Thousand) Unsolved Problems in Analysis and Partial Differential Equations' \cite{Mazya2018} a question concerning the least norm of a $W^{1,2}$-extension operator for a planar sector; see Problem 3.6. The main result of this note is a complete solution of this problem.

For $0<\alpha<2\pi$, let
\[
A_\alpha:=\{z\in\C:\ 0<\arg z<\alpha\} \,.
\]
We equip Sobolev spaces with the norm
\begin{equation}\label{eq:W12norm}
\norm{u}_{W^{1,2}(\Omega)}^2
=
\int_\Omega \left(|\nabla u|^2+|u|^2\right)dx.
\end{equation}
A linear operator
\(E:W^{1,2}(A_\alpha)\to W^{1,2}(\R^2)\) 
is called an extension operator if
\[
Eu=u\quad\text{in }A_\alpha.
\]
Maz'ya's problem consists in determining
\begin{equation}
    \label{eq:nalpha}
    \mathcal N(\alpha):= \inf \left\{ \|E\|_{W^{1,2}(A_\alpha)\to W^{1,2}(\R^2)} :\ E \ \text{is an extension operator} \right\}.
\end{equation}
The following is our main result.

\begin{theorem}\label{main}
    For any $0<\alpha<2\pi$ one has
    \begin{equation}
        \label{eq:norm}
             \mathcal N(\alpha) = \frac{1}{\min\{\sqrt{\alpha/(2\pi)},\sqrt{1-\alpha/(2\pi)}\}} \,. 
    \end{equation}
    Moreover, the infimum in \eqref{eq:nalpha} is attained by the extension operator constructed in the proof below.
\end{theorem}

\begin{remark}
Maz'ya \cite{Mazya2018} states that the right side of \eqref{eq:norm} is the least norm of an extension operator $L^{1,2}(A_\alpha)\to L^{1,2}(\R^2)$, where $L^{1,2}(\Omega)$ is the homogeneous Sobolev space consisting of $L^1_\loc(\Omega)$ functions whose distributional gradient belongs to $L^2(\Omega)$. As no proof of this statement is provided, except for the expression `Using the Radon transform', we are not sure whether in this statement $L^{1,2}(\Omega)$ is endowed with the norm in \cite[Section 2]{Mazya2018} or with the seminorm given by the $L^2$-norm of the gradient. A minor modification of our proof of Theorem \ref{main} provides a proof (without using the Radon transform) of the assertion in \cite{Mazya2018} in case the latter norm is used.

We also note that in the above we have corrected a typographical error in \cite{Mazya2018} in the expression for the least $L^{1,2}$-norm. (The constant given there is the bound for the norm from $L^{1,2}(A_\alpha)\to L^{1,2}(\R^2\setminus\overline{A_\alpha})$.)
\end{remark}


\section{The trace energy}\label{sec:proof}

In this section we consider a trace energy on sectors $A_L$ with $0<L<2\pi$. We will apply these results later both with $L=\alpha$ and $L=2\pi -\alpha$.

By the Sobolev trace theorem, there is an operator $\gamma_L:W^{1,2}(A_L)\to L^2_\loc(\partial A_L)$ such that $\gamma_L u = u|_{\partial A_L}$ for $u\in W^{1,2}(A_L)\cap C(\overline A_L)$. Let $\mathcal Q_L$ denote the range of $\gamma_L$, that is, the set of all functions $f$ on $\partial A_L$ such that there is a $u\in W^{1,2}(A_L)$ with $\gamma_Lu=f$. We are interested in the trace energy
\begin{equation}\label{eq:qLdef}
 \q_L(f)
:=
\inf\left\{
\| u \|_{W^{1,2}(A_L)}^2 :\ \gamma_L u=f
\right\}, 
\end{equation}
defined for $f\in\mathcal Q_L$. Let us summarize some basic information.

\begin{lemma}\label{existence}
    Let $0<L<2\pi$. For every $f\in\mathcal Q_L$ the infimum \eqref{eq:qLdef} is attained at a unique element $u_f\in W^{1,2}(A_L)$. Moreover, $u_f$ is a weak solution of
    $$
    \begin{cases}
        (-\Delta + 1)u_f = 0 & \text{in}\ A_L \,,\\
        u_f = f & \text{on}\ \partial A_L \,.
    \end{cases}
    $$
\end{lemma}

We include the simple proof for the sake of completeness.

\begin{proof}
    By assumption, there is a $u_*\in W^{1,2}(A_L)$ such that $\gamma_Lu_* = f$. Since $A_L$ is a Lipschitz domain, the kernel of the trace operator is $W^{1,2}_0(A_L)$. Hence
    \begin{equation}
        \label{eq:qLdefalt}
            \q_L(f) = \inf\{ \| v + u_* \|_{W^{1,2}(A_L)}^2 :\ v\in W^{1,2}_0(A_L)\}
    \end{equation}
    and $u$ is a minimizer for \eqref{eq:qLdef} if and only if $u=v+u_*$ where $v$ is a minimizer for \eqref{eq:qLdefalt}. The functional $v\mapsto \| v + u_* \|_{W^{1,2}(A_L)}^2$ is convex, coercive and weakly lower semicontinuous, so the existence of a minimizer follows by standard arguments. The uniqueness of the minimizer follows from strict convexity. Finally, the equation in the lemma arises as the Euler equation of this minimization problem.
\end{proof}

To proceed, we will slightly change our point of view. We will consider $\mathcal Q_L$ not as a subspace of $L^2_\loc(\partial A_L)$, but rather as a subset of $L^2_\loc([0,\infty),\C^2)$ by identifying a function $f \in L^2_\loc(\partial A_L)$ with a pair $(f_0,f_1)$ of functions on $[0,\infty)$, where $f_0$ and $f_1$ are the parts of $f$ on the rays $\{ \arg z = 0\}$ and $\{ \arg z = L\}$, respectively.

The following sharp comparison between two trace energies is the key ingredient in our proof of Theorem \ref{main}.

\begin{lemma}[Sharp comparison]\label{lem:comparison}
Let \(0<\alpha,\beta<2\pi\). Then $\mathcal Q_\alpha = \mathcal Q_\beta$ and for all $f$ from this set,
\begin{equation}\label{eq:comparison}
\q_\beta(f)
\le
\max\left\{\frac{\alpha}{\beta},\frac{\beta}{\alpha}\right\}\q_\alpha(f) \,.
\end{equation}
Moreover,
\begin{equation}\label{eq:comparison-sharp}
\sup_{f\ne0}\frac{\q_\beta(f)}{\q_\alpha(f)}
=
\max\left\{\frac{\alpha}{\beta},\frac{\beta}{\alpha}\right\}.
\end{equation}
\end{lemma}

We defer the proof of Lemma \ref{lem:comparison} to the next section and finish the current section by showing how it implies our main result.

\begin{proof}[Proof of Theorem \ref{main}]
    Introducing $\beta:=2\pi-\alpha$ and 
    \begin{equation}
    \label{eq:calphabeta}
    C_{\alpha,\beta}
    :=
    \max\left\{\frac{\alpha}{\beta},\frac{\beta}{\alpha}\right\},
    \end{equation}
    we can write the claimed identity \eqref{eq:norm} as
\begin{equation}
    \label{eq:norm1}
    \mathcal N(\alpha)^2 = 1+ C_{\alpha,\beta} \,.
\end{equation}

\medskip

\emph{Step 1.} In this step we prove the inequality $\leq$ in \eqref{eq:norm1} by constructing an optimal extension.

Let \(J\) denote the reversal of the two boundary components,
\[
J(f_0,f_1)=(f_1,f_0).
\]
The reversal appears because the complementary sector \(B_\alpha:=\R^2\setminus\overline{A_\alpha}\) is naturally
parametrized by \(\theta\in(\alpha,2\pi)\), so its first boundary ray is the ray
\(\theta=\alpha\) and its second boundary ray is the ray \(\theta=2\pi\), which
is the same as \(\theta=0\). Note that the trace energies are invariant under reversal:
\begin{equation}\label{eq:reversal}
\q_L(Jf)=\q_L(f).
\end{equation}
Indeed, this follows by the reflection \(\theta\mapsto L-\theta\), which preserves the \(W^{1,2}\)-norm on \(A_L\).

Let \(u\in W^{1,2}(A_\alpha)\), and let
\[
f :=\gamma_\alpha u
\]
be its trace on the two rays. By Lemma \ref{existence}, there is a unique energy minimizer in $W^{1,2}(B_\alpha)$, denoted by \(T_0u\), on \(B_\alpha\) with trace \(Jf\). The map \(T_0\) is linear, since it is the solution operator for a linear boundary-value problem.

By definition of \(\q_\beta\) and by \eqref{eq:reversal},
\[
\norm{T_0u}_{W^{1,2}(B_\alpha)}^2
=
\q_\beta(Jf)
=
\q_\beta(f).
\]
Using Lemma \ref{lem:comparison},
\[
\q_\beta(f)
\le
C_{\alpha,\beta} \,\q_\alpha(f) \,.
\]
But \(\q_\alpha(f)\) is the least energy among all functions in \(A_\alpha\)
with trace \(f\), while \(u\) is one such function. Therefore
\[
\q_\alpha(f)
\le
\norm{u}_{W^{1,2}(A_\alpha)}^2.
\]
We obtain
\begin{equation}
    \label{eq:upper}
    \norm{T_0u}_{W^{1,2}(B_\alpha)}^2
\le
C_{\alpha,\beta}\norm{u}_{W^{1,2}(A_\alpha)}^2.
\end{equation}
For $u\in W^{1,2}(A_\alpha)$ let
\[
E_0u :=
\begin{cases}
u & \text{in}\ A_\alpha,\\
T_0u & \text{in}\ B_\alpha.
\end{cases}
\]
From the fact that the trace of $T_0 u$ on $\partial B_\alpha$ coincides with the trace of $u$ on $\partial A_\alpha$, we deduce that $E_0 u$ is weakly differentiable in $\R^2$. Adding $\norm{u}_{W^{1,2}(A_\alpha)}^2$ to both sides of \eqref{eq:upper}, we see that the operator $E_0$ satisfies
$$
\| E_0 \|_{W^{1,2}(A_\alpha)\to W^{1,2}(\R^2)}^2 \leq 1+ C_{\alpha,\beta} \,.
$$
This proves the inequality $\leq$ in \eqref{eq:norm1}.

\medskip

\emph{Step 2.} We now prove the inequality $\geq$ in \eqref{eq:norm1}.

Let \(E:W^{1,2}(A_\alpha)\to W^{1,2}(\R^2)\) be any extension operator and let $0\ne f\in \mathcal Q_\alpha$. For any $u\in W^{1,2}(A_\alpha)$ with trace $f$, the function $Eu|_{B_\alpha}\in W^{1,2}(B_\alpha)$ has trace $Jf$ and therefore, recalling \eqref{eq:reversal},
\[
\norm{Eu}_{W^{1,2}(B_\alpha)}^2
\ge
\q_\beta(Jf)
=
\q_\beta(f).
\]
In particular, let us choose $u$ as the energy minimizer $u_f\in W^{1,2}(A_\alpha)$ in Lemma \ref{existence} with trace $f$. Then
\[
\norm{u_f}_{W^{1,2}(A_\alpha)}^2=\q_\alpha(f).
\]
Consequently,
\[
\frac{\norm{Eu_f}_{W^{1,2}(B_\alpha)}^2}
     {\norm{u_f}_{W^{1,2}(A_\alpha)}^2}
\ge
\frac{\q_\beta(f)}{\q_\alpha(f)} \,.
\]
Since \(Eu_f=u_f\) in \(A_\alpha\), we deduce that
\[
\| E \|_{W^{1,2}(A_\alpha)\to W^{1,2}(\R^2)}^2
\ge
\frac{\norm{Eu_f}_{W^{1,2}(\R^2)}^2}
     {\norm{u_f}_{W^{1,2}(A_\alpha)}^2}
=
1+
\frac{\norm{Eu_f}_{W^{1,2}(B_\alpha)}^2}
     {\norm{u_f}_{W^{1,2}(A_\alpha)}^2}
\ge
1+\frac{\q_\beta(f)}{\q_\alpha(f)} \,.
\]
Since \(0\ne f\in \mathcal Q_\alpha\) is arbitrary, the sharpness in Lemma \ref{lem:comparison} implies that
\[
\| E \|_{W^{1,2}(A_\alpha)\to W^{1,2}(\R^2)}^2 \geq 1 + C_{\alpha,\beta} \,.
\]
Applying this lower bound to \(E=E_0\), we also get \(\|E_0\|^2\ge1+C_{\alpha,\beta}\). Hence \(\|E_0\|^2=1+C_{\alpha,\beta}\), and the infimum is attained by \(E_0\).
This proves the inequality $\geq$ in \eqref{eq:norm1} and completes the proof of Theorem~\ref{main}.
\end{proof}


\section{Logarithmic coordinates and trace energies}

It remains to prove Lemma \ref{lem:comparison}, which is the objective of the present section. We find it convenient to work in logarithmic polar coordinates.

For \(0<L<2\pi\), let
\[
A_L :=\{re^{i\theta}:r>0,\ 0<\theta<L\}
\]
and
\[
S_L :=\R\times(0,L).
\]
Given a function \(u\) on \(A_L\), define
\begin{equation}
    \label{eq:wu}
    w(t,\theta)=u(e^t,\theta).
\end{equation}
Since
\[
dx=r\,dr\,d\theta=e^{2t}\,dt\,d\theta
\]
and
\[
|\nabla u|^2
=
|u_r|^2+\frac1{r^2}|u_\theta|^2
=
e^{-2t}\left(|w_t|^2+|w_\theta|^2\right),
\]
we obtain the identity
\begin{equation}\label{eq:polar-isometry}
\int_{A_L}\left(|\nabla u|^2+|u|^2\right)dx
=
\int_{S_L}
\left(|w_t|^2+|w_\theta|^2+e^{2t}|w|^2\right) dt\,d\theta =: \E_L(w) \,.
\end{equation}

From now on we assume only $L>0$ and note that the set $S_L$ and the quadratic form $\mathcal E_L$ are well defined for every such $L$. Let $\mathcal H_L$ denote the set of functions $w\in L^1_\loc(S_L)$ that are weakly differentiable and satisfy $\E_L(w)<\infty$. For $0<L<2\pi$, this is precisely the image of $W^{1,2}(A_L)$ under the logarithmic change of variables \eqref{eq:wu}. For arbitrary $L>0$, we use the strip definition.

For $w\in\mathcal H_L$ let 
\[
\Gamma_L w := (w(\cdot,0),w(\cdot,L))
\] 
denote the trace of \(w\) on the two boundary lines of \(S_L\). The traces are understood locally in $t$. Indeed, on every finite rectangle $(-R,R)\times(0,L)$, the condition $e^t w\in L^2$ implies $w\in L^2$; together with $w_t,w_\theta\in L^2$, this gives $w\in W^{1,2}((-R,R)\times(0,L))$. The usual trace theorem on rectangles therefore gives $w(\cdot,0),w(\cdot,L)\in L^2_\loc(\R)$.

Let $\widetilde{\mathcal Q}_L$ denote the range of $\Gamma_L$, that is, the set of pairs of functions \(g=(g_0,g_1)\) such that there is a $w\in\mathcal H_L$ with $\Gamma_Lw = g$. For $g\in\widetilde{\mathcal Q}_L$, define
\begin{equation}
\label{eq:defqtilde}
\widetilde\q_L(g)
:=
\inf\left\{
\E_L(w):\ \Gamma_Lw=g
\right\}.
\end{equation}
Clearly, if $0<L<2\pi$, then
\begin{equation}
    \label{eq:traceenergiesequiv}
    \widetilde\q_L(g) = \q_L(f)
    \qquad\text{if}\ g(t) = f(e^t) \,.
\end{equation}

Thus, Lemma \ref{lem:comparison} will follow from a corresponding comparison principle for the forms $\widetilde\q_\alpha$ and $\widetilde\q_\beta$. The comparison inequality itself is elementary, while the sharpness statement will use a semi-explicit formula for \(\widetilde\q_L(g)\) for \(L^2\)-traces. To derive this formula, we need the following simple lemma concerning the elementary minimization problem, defined for $L>0$, $\lambda\ge 0$ and \(a,b\in\C\),
\[
m_L^\lambda(a,b)
:=
\inf\left\{
\int_0^L\left(|\psi'|^2+\lambda|\psi|^2\right)d\theta:
\psi(0)=a,\ \psi(L)=b
\right\}.
\]

\begin{lemma}[Angular minimization]\label{lem:angular}
Let \(L>0\), \(\lambda\ge0\) and \(a,b\in\C\). Then
\begin{equation}\label{eq:mLdiag}
m_L^\lambda(a,b)
=
d_L^+(\lambda)\left|\frac{a+b}{\sqrt2}\right|^2
+
d_L^-(\lambda)\left|\frac{a-b}{\sqrt2}\right|^2,
\end{equation}
where
\begin{equation}\label{eq:dplusminus}
d_L^+(\lambda)
:=\sqrt\lambda\,\tanh\left(\frac{L\sqrt\lambda}{2}\right),
\qquad
d_L^-(\lambda)
:=\sqrt\lambda\,\coth\left(\frac{L\sqrt\lambda}{2}\right).
\end{equation}
At \(\lambda=0\), these formulas are interpreted as
\begin{equation}\label{eq:zero-d}
d_L^+(0)=0,
\qquad
d_L^-(0)=\frac{2}{L}.
\end{equation}
\end{lemma}

\begin{proof}
A simple compactness argument shows that there is a minimizer $\psi$ for $m_L^\lambda(a,b)$, and this minimizer solves
\[
-\psi''+\lambda\psi=0,
\qquad
\psi(0)=a,
\quad
\psi(L)=b.
\]
For \(\tau=\sqrt\lambda>0\), the unique solution of this equation is
\[
\psi(\theta)
=
\frac{a\sinh(\tau(L-\theta))+b\sinh(\tau\theta)}
     {\sinh(\tau L)}.
\]
Integration by parts and the differential equation give
\[
\int_0^L(|\psi'|^2+\lambda|\psi|^2)\,d\theta
=
\psi(L)\overline{\psi'(L)}-\psi(0)\overline{\psi'(0)},
\]
which yields 
\begin{equation}\label{eq:mLmatrix}
m_L^\lambda(a,b)
=
\tau\coth(\tau L)(|a|^2+|b|^2)
-2\tau\csch(\tau L)\operatorname{Re}(a\overline b).
\end{equation}
Diagonalizing the resulting \(2\times2\)
Hermitian form in the symmetric and antisymmetric directions \((1,1)\) and
\((1,-1)\) gives \eqref{eq:mLdiag} and \eqref{eq:dplusminus}. For
\(\lambda=0\), the minimizer is affine and
\[
m_L^0(a,b)=\frac{|a-b|^2}{L},
\]
which is exactly \eqref{eq:zero-d}.
\end{proof}

Let \(H\) be the non-negative self-adjoint operator in \(L^2(\R)\) associated
with the quadratic form
\begin{equation}\label{eq:hform}
\mathfrak h[\phi]
=
\int_\R\left(|\phi'(t)|^2+e^{2t}|\phi(t)|^2\right)dt.
\end{equation}
Its form domain consists of $\phi\in H^1(\R)$ such that $e^t\phi\in L^2(\R)$. Formally,
\[
H=-\frac{d^2}{dt^2}+e^{2t}.
\]
We shall only need below that \(\inf\sigma(H)=0\) and that zero is not an eigenvalue; these facts are proved in the sharpness argument below.

Let \(P_\lambda, \lambda\in[0,\infty),\) denote the spectral resolution of \(H\). 

\begin{lemma}[Trace-form formula]\label{lem:trace-formula}
Let \(L>0\). Let \(g=(g_0,g_1)\in \widetilde{\mathcal Q}_L\) and assume, in addition, that \(g_0,g_1\in L^2(\R)\). Then
\begin{equation}\label{eq:qLformula}
\widetilde\q_L(g)
=
\int_{[0,\infty)} d_L^+(\lambda)\,d\norm{P_\lambda g_+}_{L^2(\R)}^2
+
\int_{[0,\infty)} d_L^-(\lambda)\,d\norm{P_\lambda g_-}_{L^2(\R)}^2 \, ,
\end{equation}
where
\[
g_+ :=\frac{g_0+g_1}{\sqrt2},
\qquad
g_- :=\frac{g_0-g_1}{\sqrt2}.
\]
Conversely, if \(g_0,g_1\in L^2(\R)\) and the right side of \eqref{eq:qLformula} is finite, then \(g\in\widetilde{\mathcal Q}_L\).
\end{lemma}

\begin{proof}
\emph{Step 1.} We begin by proving the following abstract fact. Let \(A\ge0\) be a self-adjoint operator in a separable Hilbert space \(\mathfrak H\), with spectral resolution \(E_A(\lambda)\). For
\[
v \in \mathcal W_L(A)
:=
H^1((0,L);\mathfrak H)
\cap
L^2((0,L);\mathcal D(A^{1/2}))
\]
define
\[
\mathcal A_L(v)
:=
\int_0^L
\left(
\norm{v'(\theta)}_{\mathfrak H}^2
+
\norm{A^{1/2}v(\theta)}_{\mathfrak H}^2
\right)d\theta.
\]
Then, for \(u_0,u_1\in\mathfrak H\),
\begin{equation}\label{eq:abstract-trace-formula}
\begin{aligned}
&\inf_{\substack{v\in\mathcal W_L(A)\\ v(0)=u_0,\ v(L)=u_1}}
\mathcal A_L(v) \\
&\quad =
\int_{[0,\infty)}d_L^+(\lambda)\,
d\norm{E_A(\lambda)u_+}_{\mathfrak H}^2
+
\int_{[0,\infty)}d_L^-(\lambda)\,
d\norm{E_A(\lambda)u_-}_{\mathfrak H}^2,
\end{aligned}
\end{equation}
where \(u_\pm:=(u_0\pm u_1)/\sqrt2\). The equality is understood in the extended sense.

We prove \eqref{eq:abstract-trace-formula}. By the spectral theorem in the form that appears, for example, in \cite[Theorem 7.5.1]{BiSo}, \(A\) is unitarily equivalent to multiplication by the independent variable in a direct integral Hilbert space
\[
\mathfrak K=\int_{[0,\infty)}^\oplus \mathfrak K_\lambda\,d\mu(\lambda).
\]
Let \(v\in\mathcal W_L(A)\). Then \(v(\theta)\in\mathfrak H\) for every \(\theta\in[0,L]\), and we let \(V(\cdot,\theta)\in\mathfrak K\) denote the image of \(v(\theta)\) in the direct integral. We let \(U_0,U_1\) denote the images of \(u_0,u_1\). It follows that
\[
\mathcal A_L(v)
=
\int_{[0,\infty)} \int_0^L
\left(
|\partial_\theta V(\lambda,\theta)|_{\mathfrak K_\lambda}^2
+
\lambda |V(\lambda,\theta)|_{\mathfrak K_\lambda}^2
\right)d\theta\,d\mu(\lambda) \, .
\]
Here we used Tonelli's theorem. For \(\mu\)-a.e.~\(\lambda\in[0,\infty)\), we have \(V(\lambda,0)=U_0(\lambda)\) and \(V(\lambda,L)=U_1(\lambda)\). This follows from Fubini, since \(V,\partial_\theta V\in L^2\), so for \(\mu\)-a.e.~\(\lambda\) the function \(V(\lambda,\cdot)\) belongs to \(H^1((0,L);\mathfrak K_\lambda)\), with endpoint values induced by the \(H^1((0,L);\mathfrak H)\)-trace of \(v\).

The scalar formula in Lemma \ref{lem:angular} applies fiberwise also for \(\mathfrak K_\lambda\)-valued endpoints, since the minimizing path is obtained by multiplying the two endpoint vectors by the same scalar coefficient functions as in the scalar case. Applying Lemma \ref{lem:angular} in each fiber gives
\begin{align*}
& \int_0^L
\left(
|\partial_\theta V(\lambda,\theta)|_{\mathfrak K_\lambda}^2
+
\lambda |V(\lambda,\theta)|_{\mathfrak K_\lambda}^2
\right)d\theta \\
& \quad \ge
d_L^+(\lambda) \left|\frac{U_0(\lambda)+U_1(\lambda)}{\sqrt2}\right|_{\mathfrak K_\lambda}^2
+
d_L^-(\lambda)\left|\frac{U_0(\lambda)-U_1(\lambda)}{\sqrt2}\right|_{\mathfrak K_\lambda}^2.
\end{align*}
As
\[
\int_{[0,\infty)} d_L^\pm(\lambda) \left|\frac{U_0(\lambda)\pm U_1(\lambda)}{\sqrt2}\right|_{\mathfrak K_\lambda}^2 \,d\mu(\lambda) = 
\int_{[0,\infty)}d_L^\pm(\lambda)\,
d\norm{E_A(\lambda)u_\pm}_{\mathfrak H}^2 ,
\]
we arrive at the lower bound in \eqref{eq:abstract-trace-formula}.

For the upper bound, assume that the right side of \eqref{eq:abstract-trace-formula} is finite. In the same spectral representation define, for \(\lambda>0\),
\[
V(\lambda,\theta)
=
\frac{\sinh((L-\theta)\sqrt{\lambda})}{\sinh(L\sqrt{\lambda})}U_0(\lambda)
+
\frac{\sinh(\theta\sqrt{\lambda})}{\sinh(L\sqrt{\lambda})}U_1(\lambda),
\]
and for \(\lambda=0\),
\[
V(0,\theta)=\left(1-\frac{\theta}{L}\right)U_0(0)+\frac{\theta}{L}U_1(0) \, .
\]
The coefficients are bounded on \([0,L]\), so \(V\in L^2((0,L);\mathfrak K)\). Moreover, by Lemma \ref{lem:angular} and the assumed finiteness of the right side of \eqref{eq:abstract-trace-formula},
\[
\int_{[0,\infty)} \int_0^L
\left(
|\partial_\theta V(\lambda,\theta)|_{\mathfrak K_\lambda}^2
+
\lambda|V(\lambda,\theta)|_{\mathfrak K_\lambda}^2
\right)d\theta\,d\mu(\lambda)<\infty.
\]
Thus the inverse spectral transform of \(V\) belongs to \(\mathcal W_L(A)\), has boundary values \(u_0,u_1\), and attains the value on the right side of \eqref{eq:abstract-trace-formula}. This proves the abstract fact.

\medskip

\emph{Step 2.}
We now apply \eqref{eq:abstract-trace-formula} with \(\mathfrak H=L^2(\R)\) and \(A=H\). Let \(g_0,g_1\in L^2(\R)\), and let \(w\in\mathcal H_L\) be admissible with \(\Gamma_Lw=(g_0,g_1)\). Then, for a.e. \(t\),
\[
|w(t,\theta)|^2
\le
2|g_0(t)|^2+
2L\int_0^L |w_\theta(t,s)|^2\,ds.
\]
Integrating in \(t\) and \(\theta\), we obtain
\[
\int_{S_L}|w|^2\,dt\,d\theta
\le
2L\norm{g_0}_{L^2(\R)}^2
+
2L^2\int_{S_L}|w_\theta|^2\,dt\,d\theta<\infty.
\]
Thus \(w\in L^2(S_L)\). Since \(w_\theta\in L^2(S_L)\), this gives
\[
w\in H^1((0,L);L^2(\R)).
\]
Moreover, \(w_t,e^tw\in L^2(S_L)\), hence
\[
w\in L^2((0,L);\mathcal D(\mathfrak h))
=
L^2((0,L);\mathcal D(H^{1/2})).
\]
For such \(w\),
\[
\E_L(w)
=
\int_0^L
\left(
\norm{w_\theta(\cdot,\theta)}_{L^2(\R)}^2
+
\mathfrak h[w(\cdot,\theta)]
\right)d\theta.
\]
Conversely, every element of \(\mathcal W_L(H)\), considered as a function on \(S_L\), belongs to \(\mathcal H_L\) and has the same energy. Therefore, for \(L^2\)-boundary data, the minimization defining \(\widetilde\q_L\) coincides with the abstract Hilbert-space minimization problem for \(A=H\). Formula \eqref{eq:qLformula} follows from \eqref{eq:abstract-trace-formula}. The converse follows from the upper-bound construction above.
\end{proof}

Finally, we are in position to prove Lemma \ref{lem:comparison}. In view of \eqref{eq:traceenergiesequiv} it is an immediate consequence of the following result.

\begin{lemma}[Sharp comparison]\label{lem:comparison1}
Let \(\alpha,\beta>0\). Then $\widetilde{\mathcal Q}_\alpha = \widetilde{\mathcal Q}_\beta$ and for all $g$ from this set,
\begin{equation}\label{eq:comparison1}
\widetilde\q_\beta(g)
\le
\max\left\{\frac{\alpha}{\beta},\frac{\beta}{\alpha}\right\} \widetilde\q_\alpha(g) \,.
\end{equation}
Moreover,
\begin{equation}\label{eq:comparison-sharp1}
\sup_{g\ne0} \, \frac{\widetilde\q_\beta(g)}{\widetilde\q_\alpha(g)}
=
\max\left\{\frac{\alpha}{\beta},\frac{\beta}{\alpha}\right\}.
\end{equation}
\end{lemma}

\begin{proof}
Put
\[
C_{\alpha,\beta}:=\max\left\{\frac{\alpha}{\beta},\frac{\beta}{\alpha}\right\}.
\]
We first prove the comparison inequality. Let \(g\in\widetilde{\mathcal Q}_\alpha\), and let \(w\in\mathcal H_\alpha\) satisfy \(\Gamma_\alpha w=g\). Define
\[
v(t,\theta):=w\left(t,\frac{\alpha}{\beta}\theta\right),
\qquad 0<\theta<\beta.
\]
Then \(\Gamma_\beta v=g\). Moreover, after the change of variables \(s=\alpha\theta/\beta\),
\[
\E_\beta(v)
=
\frac{\beta}{\alpha}
\iint_{S_\alpha}
\left(|w_t|^2+e^{2t}|w|^2\right)\,dt\,ds
+
\frac{\alpha}{\beta}
\iint_{S_\alpha}|w_\theta|^2\,dt\,ds.
\]
Hence
\[
\E_\beta(v)
\le
C_{\alpha,\beta}\E_\alpha(w).
\]
Taking the infimum over all such \(w\) gives
\[
\widetilde\q_\beta(g)
\le
C_{\alpha,\beta}\widetilde\q_\alpha(g).
\]
In particular, \(\widetilde{\mathcal Q}_\alpha\subset \widetilde{\mathcal Q}_\beta\). Interchanging \(\alpha\) and \(\beta\) gives the reverse inclusion, and therefore \(\widetilde{\mathcal Q}_\alpha=\widetilde{\mathcal Q}_\beta\).

It remains to prove sharpness, that is, \eqref{eq:comparison-sharp1}. For this part we use the \(L^2\)-trace formula from Lemma \ref{lem:trace-formula}. For \(\lambda>0\), write \(\tau=\sqrt\lambda\). By Lemma \ref{lem:angular},
\begin{equation}\label{eq:Rplus}
R_+(\lambda)
:=
\frac{d_\beta^+(\lambda)}{d_\alpha^+(\lambda)}
=
\frac{\tanh(\tau\beta/2)}{\tanh(\tau\alpha/2)}
\end{equation}
and
\begin{equation}\label{eq:Rminus}
R_-(\lambda)
:=
\frac{d_\beta^-(\lambda)}{d_\alpha^-(\lambda)}
=
\frac{\tanh(\tau\alpha/2)}{\tanh(\tau\beta/2)}
=
\frac1{R_+(\lambda)} \,.
\end{equation}

We recall that the bottom of the spectrum of
\[
H=-\frac{d^2}{dt^2}+e^{2t}
\]
is zero. Indeed, let \(0\ne\eta\in C_c^\infty((-2,-1))\) and put
\[
\eta_R(t)=R^{-1/2}\eta(t/R),
\qquad R>1.
\]
Then \(\eta_R\) is supported in \((-2R,-R)\), and
\[
\frac{\mathfrak h[\eta_R]}{\norm{\eta_R}_{L^2(\R)}^2}
\le
\frac{C}{R^2}+e^{-2R}\to0.
\]
Thus \(\inf\sigma(H)=0\). Moreover, zero is not an eigenvalue, since
\(\mathfrak h[\phi]=0\) would force \(\phi'=0\) and \(e^t\phi=0\), hence
\(\phi=0\). Therefore, for every \(\delta>0\), the spectral projection of \(H\)
on \([0,\delta]\) is nonzero and has no mass at \(0\). Here \(P(I)\) denotes the spectral projection of \(H\) corresponding to the Borel set \(I\).

If \(\beta>\alpha\), then
\[
\lim_{\lambda\downarrow0}R_+(\lambda)=\frac{\beta}{\alpha}.
\]
Given \(\eps>0\), choose \(\delta>0\) so small that
\[
R_+(\lambda)\ge \frac{\beta}{\alpha}-\eps
\qquad\text{for }0<\lambda\le\delta.
\]
Let \(0\ne h\in\operatorname{Ran} P([0,\delta])\) and choose
\[
g_0:=g_1:=h \,.
\]
By the converse part of Lemma \ref{lem:trace-formula}, this \(g\) belongs to both \(\widetilde{\mathcal Q}_\alpha\) and \(\widetilde{\mathcal Q}_\beta\). Since \(g_-=0\), the quotient \(\widetilde\q_\beta(g)/\widetilde\q_\alpha(g)\) is a weighted average of \(R_+(\lambda)\) over \((0,\delta]\). Hence
\[
\frac{\widetilde\q_\beta(g)}{\widetilde\q_\alpha(g)}
\ge
\frac{\beta}{\alpha}-\eps.
\]
Since \(\eps>0\) is arbitrary, this gives sharpness when \(\beta>\alpha\).

If \(\alpha>\beta\), then
\[
\lim_{\lambda\downarrow0}R_-(\lambda)=\frac{\alpha}{\beta}.
\]
The same argument with
\[
g_0:=h,
\qquad
g_1:=-h
\]
gives sharpness in this case. The case \(\alpha=\beta\) is immediate. This proves \eqref{eq:comparison-sharp1}.
\end{proof}

\begin{remark}
    Our proof yields the following generalization of Theorem \ref{main}. Let $0\leq w\in L^1_\loc((0,\infty))$ such that the point $0$ belongs to the spectrum of the operator
    $$
    -\frac{d^2}{dt^2} + e^{2t} w(e^t)
    \qquad\text{in}\ L^2(\R) \,,
    $$
    but is not an eigenvalue. Defining the (semi-)norm
    $$
    \norm{u}_{W^{1,2}_w(\Omega)}^2
=
\int_\Omega \left(|\nabla u|^2+w(|x|)|u|^2\right)dx \,,
    $$
    we find that the infimum of $\|E\|_{W^{1,2}_w(A_\alpha)\to W^{1,2}_w(\R^2)}$ over all extension operators is given by the right side of \eqref{eq:norm}. For $w\equiv 1$, this is Theorem \ref{main}, while for $w\equiv 0$, this is the assertion in \cite{Mazya2018}.
\end{remark}


\section*{Acknowledgments}

R.~L.~F.~acknowledges partial support from US NSF grant DMS-1954995 and the DFG grants EXC-2111-390814868 TRR 352-Project-ID 470903074, and FR 2664/3-1. P.~I.~acknowledges partial support from the US NSF CAREER grant DMS-2152401, US NSF grant DMS-2554183,  a Simons Fellowship, and a Humboldt Research Fellowship for Experienced Researchers. The authors acknowledge the use of AI tools during the exploratory stage of this project. All mathematical arguments and proofs in the final manuscript were checked and written by the authors.


\end{document}